\documentclass[11pt]{amsart}
\usepackage[usenames,dvipsnames,svgnames,table]{xcolor}
\usepackage{amssymb,amsfonts,amsmath,amscd,stmaryrd}
\usepackage{verbatim}
\usepackage[all]{xy}
\usepackage{hyperref}

\newtheorem{theorem}{Theorem}
\newtheorem{lemma}[theorem]{Lemma}

\newtheorem*{theorem*}{Theorem}
\newtheorem*{corollary*}{Corollary}

\theoremstyle{definition}
\newtheorem*{remark*}{Closing Remark}

\newcommand{\mbz}{\mbox{$\mathbb{Z}$}}

\newcommand{\mbc}{\mbox{$\mathbb{C}$}}
\newcommand{\mbk}{\mbox{$\mathbb{K}$}}

\newcommand{\mfm}{\mbox{$\mathfrak{m}$}}
\newcommand{\mfn}{\mbox{$\mathfrak{n}$}}
\newcommand{\mfp}{\mbox{$\mathfrak{p}$}}

\newcommand{\grade}{\mbox{${\rm grade}$}}

\newcommand{\length}{\mbox{${\rm length}$}}
\newcommand{\rank}{\mbox{${\rm rank}$}}
\newcommand{\pd}{\mbox{${\rm proj}\,{\rm dim}\,$}}
\newcommand{\gldim}{\mbox{${\rm gl}\,{\rm dim}\,$}}
\newcommand{\wdim}{\mbox{${\rm w}.{\rm dim}\,$}}
\newcommand{\krulldim}{\mbox{${\rm Krull}\,{\rm dim}\,$}}
\newcommand{\Ext}{\mbox{${\rm Ext}$}}

\newcommand{\Ann}{\mbox{${\rm Ann}$}}

\newcommand{\Max}{\mbox{${\rm Max}$}}

\title[Noetherian rings of low global dimension]
{Noetherian rings of Low global dimension and syzygetic prime ideals}

\author{Francesc Planas-Vilanova}

\date{\today}


\keywords{Global dimension, Noetherian regular rings, ideal of linear
  type, syzygetic ideal.\\ This work is partially supported by the
  Catalan grant 2014 SGR-634}% and grant MTM2015-69135-P}

\begin{document}

\begin{abstract}
Let $R$ be a Noetherian ring. We prove that $R$ has global dimension
at most two if, and only if, every prime ideal of $R$ is of linear
type. Similarly, we show that $R$ has global dimension at most three
if, and only if, every prime ideal of $R$ is syzygetic. As a
consequence, one derives a characterization of these rings using the
Andr\'e-Quillen homology.
\end{abstract}

\maketitle 

Let $R$ be a commutative ring. An ideal $I$ of $R$ is said to be of
linear type if the graded surjective morphism $\alpha:{\bf S}(I)\to
{\bf R}(I)$, from the symmetric algebra of $I$ to the Rees algebra of
$I$, is an isomorphism; $I$ is said to be syzygetic if the second
component $\alpha_2:{\bf S}_2(I)\to I^2$ is an isomorphism. It is
known that $R$ has weak dimension at most one if, and only if, every
ideal of $R$ is of linear type, and equivalently if, and only if,
every ideal of $R$ is syzygetic. In particular, rings of weak
dimension at most one are characterized in terms of the
Andr\'e-Quillen homology (see \cite{planasPams}).

Recall that the weak dimension of a ring $R$, denoted by $\wdim(R)$,
is the supremum of the flat dimensions of all $R$-modules and that the
global dimension of $R$, denoted by $\gldim(R)$, is the supremum of
the projective dimensions of all $R$-modules. Clearly
$\wdim(R)\leq\gldim(R)$, and when $R$ is Noetherian, they
agree. Since $\gldim(R)=\sup\{\gldim(R_{\mathfrak{m}})\mid
\mfm\in\Max(R)\}$, then, for a Noetherian ring $R$, $\gldim(R)\leq N$
is equivalent to $R_{\mathfrak{m}}$ being regular local with
$\krulldim(R_{\mathfrak{m}})\leq N$, for every maximal ideal $\mfm$ of
$R$ (see, e.g., \cite{lam}).

The purpose of this note is to extend these characterizations of rings
of $\wdim(R)\leq 1$ to rings of global dimension at most two and
three, but now in the Noetherian context. This is done in quite
similar terms. Concretely, we prove the theorem below. Item~(A), shown
in general in \cite{planasPams}, is included here just for the sake of
completeness.

\begin{theorem*}
Let $R$ be a Noetherian ring.
\begin{itemize}
\item[{(A)}:] $\gldim(R)\leq 1\Leftrightarrow$ every ideal of $R$ is
  of linear type $\Leftrightarrow $ every ideal of $R$ is syzygetic.
\item[{(B)}:] $\gldim(R)\leq 2\Leftrightarrow$ every prime ideal of
  $R$ is of linear type.
\item[{(C)}:] $\gldim(R)\leq 3\Leftrightarrow$ every prime ideal of
  $R$ is syzygetic.
\end{itemize}
\end{theorem*}

Since the linear type and syzygetic conditions are clearly local, to
prove this theorem, one can suppose that $(R,\mfm,k)$ is a Noetherian
local ring with maximal ideal $\mfm$ and residue field $k$. Moreover,
one can substitute the condition $\gldim(R)\leq N$ by the condition
``$R$ is regular local of $\krulldim(R)\leq N$''. If $\krulldim(R)\leq
1$, every nonzero proper ideal is generated by a nonzero divisor,
hence of linear type and syzygetic (see, e.g.,
\cite[Corollary~3.7]{hsv}).  Suppose that $\krulldim(R)\leq 2$ or
3. Since $R$ is regular local, then $\mfm$ is generated by an
$R$-regular sequence, so $\mfm$ is of linear type (see, e.g.,
\cite[Corollary~3.8]{hsv}); moreover $R$ is a UFD, thus every height
one prime ideal is principal (generated by a nonzero divisor), and so
again of linear type; if $\krulldim(R)\leq 3$, then every height two
prime ideal is perfect and generically a complete intersection, hence
syzgygetic (see, e.g., \cite[Remark~page~91]{hsv}). This shows the
``only if'' implications.

Observe that the proof of \cite[Corollary~3.8]{hsv} shows that a
Noetherian local ring with syzygetic maximal ideal is
regular. Therefore, in order to prove the ``if'' implication in
Theorem~(A), it is enough to display, in a two dimensional regular
local ring, a non syzygetic ideal. Similarly, to prove prove the
``if'' implication in Theorem~(B), it suffices to exhibit, in a three
dimensional regular local ring, a height two prime ideal which is not
of linear type. Finally, to prove the ``if'' implication in
Theorem~(C), we exhibit, in a four dimensional regular local ring, a
height three prime ideal which is not syzygetic.

In this direction, we show the next result, a kind of rephrasing of
\cite[Lemma~3]{planasPams} but under the regular local hypothesis, and
hence easier to prove it. We give a different alternative proof using
\cite[Corollary~4.11]{planasCamb}.

\begin{lemma}\label{lLocal2}
Let $(R,\mfm,k)$ be a regular local ring of Krull dimension $2$. Let
$x,y$ be a regular system of parameters. Then $\mfm^2$ is a non
syzygetic ideal. 
\end{lemma}
\begin{proof}
Let $I=\mfm^2=(x^2,xy,y^2)$ and $J=(x^2,y^2)$. Using that $x,y$ is an
$R$-regular sequence, it is easy to check that $xy\not\in J$ (see
\cite[Theorem~9.2.2]{bh}, to relate it to the simplest case of the
Monomial Conjecture). Since $(xy)^2\in JI$, then $J:xy\subsetneq
JI:(xy)^2$. By \cite[Corollary~4.11]{planasCamb}, we conclude that $I$
is not syzygetic.
\end{proof}

Next lemma displays, in a three dimensional regular local ring, a
height two prime ideal which is not of linear type, thus generalizing
\cite[Corollary~2.7]{gop}. There, in the context of the Shimoda
Conjecture, one exhibits, in a three dimensional regular local ring, a
non-complete intersection height two prime ideal. (Recall that
complete intersection implies linear type.)

\begin{lemma}\label{lLocal3}
Let $(R,\mfm,k)$ be a regular local ring of Krull dimension $3$. Let
$x,y,z$ be a regular system of parameters. Let $I$ be the ideal of $R$
generated by
\begin{eqnarray*}
  &&f_1=y^3-x^4\mbox{ , }f_2=xyz-z^3+x^4-xy^3,\\
  &&f_3=x^2y+y^2z-xz^2-x^3y\mbox{ and }f_4=xy^2-yz^2-x^2y^2+x^3z.
\end{eqnarray*}
Then $I$ is a height two prime ideal minimally generated by four
elements. In particular, $I$ is not of linear type.
\end{lemma}

Finally, the third lemma exhibits a non syzygetic height three prime
ideal in a four dimensional regular local ring.

\begin{lemma}\label{lLocal4}
Let $(R,\mfm,k)$ be a regular local ring of dimension $4$. Let
$x,y,z,t$ be a regular system of parameters. Let $I$ be the ideal of
$R$ generated by
\begin{eqnarray*}
&&  f_1=yz-xt\mbox{ , }f_2=z^3-x^5\mbox{ , }f_3=z^2t-x^4y\mbox{ ,}\\
&&  f_4=zt^2-x^3y^2\mbox{ , }f_5=t^3-x^2y^3\mbox{ , }f_6=y^4-x^5\mbox{ ,}\\
&&  f_7=y^3t-x^4z\mbox{ , }f_8=y^2t^2-x^3z^2.
\end{eqnarray*}
Then $I$ is a height three prime ideal which is not a syzygetic ideal.
\end{lemma}

The general skeleton of the proofs of Lemmas~\ref{lLocal3} and
\ref{lLocal4} are similar to that of the proof of
\cite[Proposition~2.6]{gop}. Namely, once the candidate $I$ is chosen,
we show that $I$ is perfect with the desired height, in particular,
height unmixed. Then we pick an associated prime $\mfp$ to $I$, which
will be of the same height, and, by means of multiplicity theory, we
show that $xR+I$ and $xR+\mfp$ have the same colength, concluding, by
Nakayama's Lemma, that $I$ and $\mfp$ are equal.
  
The ideal $I$ displayed in Lemma~\ref{lLocal3} is a small variation of
\cite[Example~3.7]{huneke}. Concretely, Huneke considers the height
two prime ideal defined by the kernel of the homomorphism from the
power series ring $\mbc\llbracket X,Y,Z\rrbracket$ to $\mbc\llbracket
t\rrbracket$, sending $X,Y,Z$ to $t^6,t^7+t^{10},t^8$,
respectively. He shows that this ideal is generated by the $3\times 3$
minors of a specified $4\times 3$ matrix $L$, whose entries are either
$0$, or else among one of the monomial terms $X,Y,Z,X^2, XY$ times a
$\pm 1,\pm 2$ integer coefficient. Our example consists in taking
these $3\times 3$ minors, but substituting the variables $X,Y,Z$ by
the regular parameters $x,y,z$ and, in order to avoid characteristic
two problems, replacing in $L$, $\pm 2$ by $\pm 1$ (surprisingly
enough, it works).

As for the ideal $I$ considered in Lemma~\ref{lLocal4}, we recover a
particular case of a family of prime ideals with unbounded number of
generators provided by Bresinsky in \cite{bresinsky}. Concretely, we
consider the kernel of the homomorphism from $\mbk[X,Y,Z,T]$, $\mbk$
any field, to $\mbk[t]$, sending $X,Y,Z,T$ to
$t^{12},t^{15},t^{20},t^{23}$ and then, as before, just substitute the
variables by the regular parameters.

Before proceeding to prove Lemmas~\ref{lLocal3} and ~\ref{lLocal4}, we
highlight the good behaviour of the syzygetic and linear type
conditions through faifhfully flat morphisms of rings. Indeed, this
follows from \cite[Corollaire~2.3]{planasMan} (see also
\cite[Theorem~2.4 and Example~2.3]{planasCamb}), where one shows that
these conditions are characterized in terms of the exactness of a
complex of $R$-modules and noting that, if $R\to S$ is a faithfully
flat morphism of rings, then $I\otimes _RS=IS$.

\begin{proof}[Proof of Lemma~\ref{lLocal3}]
Since $(R,\mfm)$ is a three dimensional regular local ring with
maximal ideal $\mfm$ generated by $x,y,z$, then its completion
$(\widehat{R},\widehat{\mfm})$ is a three dimensional regular local
ring with maximal ideal
$\widehat{\mfm}=\mfm\widehat{R}=(x,y,z)\widehat{R}$ generated by the
regular system of paremeters $x,y,z$. Let $I=(f_1,f_2,f_3,f_4)$ and
$\widehat{I}=I\widehat{R}=(f_1,f_2,f_3,f_4)\widehat{R}$. If we prove
that $\widehat{I}$ is prime and not of linear type, then
$I=I\widehat{R}\cap R$ is prime and not of linear type, because the
completion morphism is faithfully flat (see, e.g.,
\cite[\S~8]{matsumura}). Therefore we can suppose that $R$ is
complete.

First observe that $f_1,f_2,f_3,f_4$ are, up to a change of sign, the
$3\times 3$ minors of the matrix
\begin{eqnarray*}
  \varphi_2=\left(\begin{array}{rcr}x&xy&z\\x&y&0\\-z&-x^2&-y\\-y&-z&x
  \end{array}\right).
\end{eqnarray*}
In other words, $I=I_3(\varphi_2)$. Since $(f_1,f_2,x)=(x,y^3,z^3)$,
then $\grade(f_1,f_2,x)=3$. By \cite[Corollary~1.6.19]{bh}, $f_1,f_2$
is an $R$-regular sequence in $I_3(\varphi_2)$ and so
$\grade(I_3(\varphi_2))\geq 2$. Let $\varphi_1$ be the $1\times 4$
matrix defined as $(f_1,\ldots,f_4)$. By the Hilbert-Burch Theorem
(e.g., \cite[Theorem~1.4.16]{bh}),
\begin{eqnarray*}
0\rightarrow
F_2=R^3\xrightarrow{\varphi_2}F_1=R^4\xrightarrow{\varphi_1}F_0=R\rightarrow
R/I\rightarrow 0
\end{eqnarray*}
is a free resolution of $R/I$. (It is minimal since
$\varphi_2(R^3)\subset\mfm R^4$ and $\varphi_1(R^4)=I\subset \mfm$.)
Therefore
\begin{eqnarray*}
2\leq\grade(I)=\min\{i\geq 0\mid \Ext^{i}_{R}(R/I,R)\neq
0\}\leq\pd_R(R/I)\leq 2
\end{eqnarray*}
and $I$ is a perfect ideal of grade $2$ (see, e.g.,
\cite[Theorem~1.2.5 and page~25]{bh}). In particular, $I$ is grade
(and height) unmixed (see, e.g., \cite[Proposition~1.4.15]{bh}) and so
$\mfm$ is not an associated prime to $I$.

Let $\mfp$ be any associated prime of $I$ and set $D=R/\mfp$. Thus $D$
is a one dimensional complete Noetherian local domain
(\cite[page~63]{matsumura}). Let $V$ be its integral closure in its
quotient field $K$. Then $V$ is a one dimensional integrally closed
Noetherian local domain, hence a DVR, a discrete valuation ring; note
that $V$ is also complete (see, e.g., \cite[Theorem~4.3.4]{sh}).

Let $\nu$ be the valuation on $K$ corresponding to $V$. Keep calling
$x,y,z$ to the images of the regular system of parameters of $R$ in
$V$. Set $\nu_x=\nu(x)$, $\nu_y=\nu(y)$ and $\nu_z=\nu(z)$. In $V$,
$f_1=0$. Applying $\nu$ to the equality $x^4-y^3=0$, one gets
$4\nu_x=3\nu_y$. Thus $\nu_x=3q$, for some integer $q\geq 1$. In fact,
$q>1$. Indeed, suppose that $q=1$, $\nu_x=3$ and $\nu_y=4$. Since
$f_2=0$ in $V$, then $z^3=x(yz+x^3-y^3)$. Applying $\nu$ to this
equality, $3\nu_z\geq\min(12,7+\nu_z)$, which implies $\nu_z\geq
4$. Since $f_3=0$ in $V$, then $x^2y=-y^2z+xz^2+x^3y$. Applying $\nu$
to this equality, one gets $10\geq\min(12,11,13)$, a
contradiction. Therefore $\nu_x\geq 6$.

Observe that $xR+I=(x,y^3,y^2z,yz^2,z^3)$. Set $S=R/xR$ and consider
(by abuse of notation) $y,z$ a regular system of parameters of the
regular local ring $(S,\mfn)$, where $\mfn=(y,z)$. Then $R/(xR+I)\cong
S/\mfn^3$. Since $xR=\Ann_R(S)$, then
$\length_R(R/(xR+I))=\length_S(S/\mfn^3)$. Since $y,z$ is a
$S$-regular sequence, there exists a graded isomorphism $k[Y,Z]\cong
G(\mfn)$ of $k$-algebras, between the polynomial ring in two variables
$Y,Z$ over the field $k=S/\mfn$ and the associated graded ring of the
ideal $\mfn$. Using the two exact short sequences
$0\to\mfn^i/\mfn^{i+1}\to S/\mfn^{i+1}\to S/\mfn^i\to 0$, for $i=1,2$,
one deduces that $\length_S(S/\mfn^3)=6$. Therefore
$\length_R(R/(xR+I))=6$.

On the other hand, since $xR+I\subseteq xR+\mfp$ and $R/(xR+\mfp)\cong
(R/\mfp)/(x\cdot R/\mfp)=D/xD$, then
\begin{eqnarray*}
6=\length_R(R/(xR+I))\geq\length_R(R/(xR+\mfp))=\length_D(D/xD).
\end{eqnarray*}
Since $f_1=0$ and $f_2=0$ in $D$, then $y^3,z^3\in xD$, and so $xD$ is
an $\mfm/\mfp$-primary ideal of the one dimensional Cohen-Macaulay
local domain $(D,\mfm/\mfp,k)$. Using \cite[Proposition~11.1.10]{sh},
we deduce that $\length_D(D/xD)=e_D(xD;D)$, the multiplicity of $xD$
on $D$. Clearly, $V$ is a finitely generated Cohen-Macaulay $D$-module
of $\rank_D(V)=1$. By \cite[Corollary~4.6.11]{bh},
$\length_D(V/xV)=e_D(xD;D)\cdot\rank_D(V)=e_D(xD;D)$. Note that
$\length_D(V/xV)=[k_V:k]\cdot\length_V(V/xV)$, where $[k_V:k]$ is
the degree of the extension of the residue fields of $V$ and of
$D$. Since $V$ is a DVR, $\length_V(V/xV)=\nu(x)=\nu_x$. Therefore,
$\length_D(D/xD)=[k_V:k]\cdot\nu_x$. Summing up all together, we get
\begin{eqnarray*}
6=\length_R(R/(xR+I))\geq\length_R(R/(xR+\mfp))=[k_V:k]\cdot\nu_x\geq
6.
\end{eqnarray*}
Therefore $\length_R(R/(xR+I))=\length_R(R/(xR+\mfp))$ and, by the
additivity of the length with respect to exact short sequences,
$xR+I=xR+\mfp$.

Note that $x\not\in\mfp$, otherwise $\mfp\supset xR+I\supset
(x,y^3,z^3)$ and $\mfp=\mfm$, a contradiction. Then $\mfp\cap
xR=x\mfp$. In particular, on tensoring $0\to\mfp/I\to R/I\to R/\mfp\to
0$ by $R/xR$, one obtains the exact sequence $0\to L/xL\to R/(xR+I)\to
R/(xR+\mfp)\to 0$, where $L=\mfp/I$. Since $xR+I=xR+\mfp$, then
$L=xL$. By Nakayama's Lemma, $L=0$ and $I=\mfp$.

We conclude that $I$ is a prime ideal of $R$. Since the aforementioned
resolution of $R/I$ is minimal, $I$ is minimally generated by $4$
elements, which in particular implies that $I$ is not of linear type,
because the minimal number of generators of an ideal of linear type is
bounded above by the dimension of the ring (see
\cite[Proposition~2.4]{hsv}).
\end{proof}

\begin{proof}[Proof of Lemma~\ref{lLocal4}]
Since the proof of the present proposition is quite analogous to that
of Lemma~\ref{lLocal3}, we skip now some details and re-direct the
reader to the former proof.  For instance, as before, we can suppose
that $R$ is complete.  Let $\varphi_1$ be the $1\times 8$ matrix
defined as $(f_1,\ldots ,f_8)$. Let $\varphi_2$ and $\varphi_3$ be the
matrices defined as:
\begin{eqnarray*}
\varphi_2=\left( \begin{array}{cccccccccccc}
y^2t&y^3&t^2&zt&z^2&x^3z&x^4&-yt^2&x^2y^2&x^3y&x^4&0\\   
0&  0& 0& 0& -y&0&  0& x^3&  0&   0&  -t&0\\   
0&  0& 0& -y&x& 0&  0& 0&   0&   -t& z& x^3\\  
0&  0& -y&x& 0& 0&  0& 0&   -t&  z&  0& 0\\   
0&  0& x& 0& 0& 0&  0& -xy& z&   0&  0& -y^2\\ 
0&  -z&0& 0& 0& 0&  -t&-x^3& 0&   0&  0& -x^2y\\
-z& x& 0& 0& 0& -t& y& 0&   0&   0&  0& 0\\   
x&  0& 0& 0& 0& y&  0& z&   0&   0&  0& t
\end{array}\right)\mbox{ and }
\end{eqnarray*}
\begin{eqnarray*}
\varphi_3=\left(\begin{array}{ccccc}
-t& 0& 0& -y&0\\  
0&  0& 0& t& -x^2y\\
0&  0& -z&0& -yt\\ 
0&  -z&t& 0& 0\\   
-x^3&t& 0& 0& 0\\   
z&  0& 0& x& 0\\   
0&  0& 0& -z&x^3\\  
-y& 0& 0& 0& -t\\  
0&  0& x& 0& y^2\\  
0&  x& -y&0& 0\\   
0&  -y&0& 0& -x^3\\ 
x&  0& 0& 0& z
\end{array}\right).
\end{eqnarray*}
Since $\varphi_3\cdot\varphi_2=0$ and $\varphi_2\cdot\varphi_1=0$, then
\begin{eqnarray*}
  0\rightarrow
  F_3=R^5\xrightarrow{\varphi_3}F_2=R^{12}\xrightarrow{\varphi_2}
  F_1=R^{8}\xrightarrow{\varphi_1}F_0=R\rightarrow R/I\rightarrow 0
\end{eqnarray*}
is a complex of $R$-modules. To prove its exactness we will use the
acyclicity criterion of Buchsbaum and Eisenbud (see, e.g.,
\cite[Theorem~1.4.12]{bh}). Set $r_i=\sum_{j=i}^3(-1)^{j-i}\rank F_j$,
so that $r_1=1$, $r_2=7$ and $r_3=5$.

Note that, since $(f_2,f_5,f_6,x)=(x,y^4,z^3,t^3)$, then
$\grade(f_2,f_5,f_6,x)=4$. By \cite[Corollary~1.6.19]{bh},
$f_2,f_5,f_6$ is an $R$-regular sequence in $I=I_1(\varphi_1)$. In
particular, $\grade(I)\geq 3$.

In order to prove $\grade(I_7(\varphi_2))\geq 2$, we look for minors
of $\varphi_2$ with pure terms in one of the parameters. For instance,
up to sign, the minor $g_1:=y^{10}-2x^5y^6+x^{10}y^2\in
I_7(\varphi_2)$, with pure term in $y$, is obtained from the $7\times
7$ submatrix given by the rows $1,2,3,4,5,7,8$ and the columns
$2,3,4,5,6,7,12$. Similarly, we get $g_2:=z^8-2x^5z^5+x^{10}z^2\in
I_7(\varphi_2)$ from the $7\times 7$ submatrix given by the rows
$1,3,4,5,6,7,8$ and the columns $1,2,5,8,9,10,11$. Since
$(g_1,g_2,x)=(x,y^{10},z^8)$, then $\grade(g_1,g_2,x)=3$ and $g_1,g_2$
is an $R$-regular sequence in $I_7(\varphi_2)$ and
$\grade(I_7(\varphi_2))\geq 2$.

As before, let us seek for minors of $\varphi_3$ with pure terms in
one of the parameters. Thus $h_1=y^6-x^5y^2\in I_5(\varphi_3)$ is
obtained from the $5\times 5$ submatrix given by the rows
$1,8,9,10,11$; $h_2=z^5-x^5z^2\in I_5(\varphi_3)$ is obtained from the
rows $3,4,6,7,12$ and, finally, $h_3=t^5-x^2y^3t^2\in I_5(\varphi_3)$
is obtained from the rows $1,2,4,5,8$. Since
$(h_1,h_2,h_3,x)=(x,y^6,z^5,t^5)$, then $h_1,h_2,h_3$ is an
$R$-regular sequence in $I_5(\varphi_3)$ and
$\grade(I_5(\varphi_3))\geq 3$. We conclude that the complex above is
a (minimal) free resolution of $R/I$.

Therefore $I$ is a perfect ideal of grade $3$. In particular, $I$ is
height unmixed and so $\mfm$ is not an associated prime to $I$.

Let $\mfp$ be any associated prime of $I$ and set $D=R/\mfp$. Thus $D$
is a one dimensional complete Noetherian local domain. Then $V$, the
integral closure of $D$ in its quotient field $K$, is a DVR (see
\cite[Theorem~4.3.4]{sh}). Let $\nu$ be the valuation on $K$
corresponding to $V$. Set $\nu_x=\nu(x)$, $\nu_y=\nu(y)$,
$\nu_z=\nu(z)$ and $\nu(t)=\nu_t$. In $V$, $f_2=z^3-x^5=0$,
$f_5=t^3-x^2y^3=0$ and $f_6=y^4-x^5=0$. Applying $\nu$ to these
equalities, one gets $3\nu_z=5\nu_x$, $3\nu_t=2\nu_x+3\nu_y$ and
$4\nu_y=5\nu_x$. The positive vector
$(\nu_x,\nu_y,\nu_z,\nu_t)\in\mbz^4$, with smallest $\nu_x\geq 1$,
satistying these three conditions is $(12,15,29,23)$. (Clealy, this
vector also satisfies all the other conditions arising from $f_i=0$.)
In particular, $\nu_x\geq 12$.

Let $(S,\mfn,k)$ be the regular local ring with $S=R/xR$ and
$\mfn=\mfm/xR=(y,z,t)$, by abuse of notation.  One has
$xR+I=(x,yz,z^3,z^2t,zt^2,t^3,y^4,y^3t,y^2t^2)$ and $R/(xR+I)\cong
S/J$, where $J$ is the ideal of $S$ defined as
$J=(yz,z^3,z^2t,zt^2,t^3,y^4,y^3t,y^2t^2)$. Since $xR=\Ann_R(S)$, then
$\length_R(R/(xR+I))=\length_S(S/J)$. Since $y,z,t$ is a $S$-regular
sequence, there exists a graded isomorphism $k[Y,Z,T]\cong G(\mfn)$ of
$k$-algebras, where $G=G(\mfn)$ stands for the associated graded ring
of $\mfn$. Let $J^*$ denote the homogeneous ideal of $G$ generated by
all the initial forms of elements of $J$. Proceeding as in the proof
of \cite[Lemma~2.9]{gop}, one sees that
$\length_S(S/J)=\length_S(G/J^*)$. Let $L$ be the ideal of $G$
generated by the initial forms of
$yz,z^3,z^2t,zt^2,t^3,y^4,y^3t,y^2t^2$ in $G$. By
\cite[Theorem~2.11]{gop} (see also \cite[Remark~2.10]{gop}),
$L=J^*$. Hence, $\length_S(G/J^*)=\length_S(G/L)$. Through the
isomorsphim $k[Y,Z,T]\cong G$, one deduces that $G/L$ is isomorphic to
the $k$-vector space spanned by
$1,Y,YT,YT^2,Y^2,Y^2T,Y^3,Z,ZT,Z^2,T,T^2$. Therefore
$\length_R(R/(xR+I))=\length_S(S/J)=\length_S(G/J^*)=\length_S(G/L)=12$.

As in Lemma~\ref{lLocal3}, since $xR+I\subseteq xR+\mfp$ and
$R/(xR+\mfp)\cong (R/\mfp)/(x\cdot R/\mfp)=D/xD$, then
\begin{eqnarray*}
12=\length_R(R/(xR+I))\geq\length_R(R/(xR+\mfp))=\length_D(D/xD).
\end{eqnarray*}
Since $f_6=0$, $f_2=0$ and $f_5=0$ in $D$, then $y^4,z^3,t^3\in xD$,
and so $xD$ is an $\mfm/\mfp$-primary ideal of the one dimensional
Cohen-Macaulay local domain $(D,\mfm/\mfp,k)$. Thus
$\length_D(D/xD)=e_D(xD;D)$. Since $V$ is a finitely generated
Cohen-Macaulay $D$-module of rank $1$,
$\length_D(V/xV)=e_D(xD;D)\cdot\rank_D(V)=e_D(xD;D)$. Moreover
$\length_D(V/xV)=[k_V:k]\cdot\length_V(V/xV)=[k_V:k]\cdot\nu(x)$
Therefore, $\length_D(D/xD)=[k_V:k]\cdot\nu_x$. Summing up all
together,
\begin{eqnarray*}
12=\length_R(R/(xR+I))\geq\length_R(R/(xR+\mfp))=[k_V:k]\cdot\nu_x\geq
12.
\end{eqnarray*}
Therefore $\length_R(R/(xR+I))=\length_R(R/(xR+\mfp))$ and
$xR+I=xR+\mfp$.

Again, note that $x\not\in\mfp$, otherwise $\mfp\supset xR+I\supset
(x,y^4,z^3,t^3)$ and $\mfp=\mfm$, a contradiction. So $\mfp\cap
xR=x\mfp$. Proceeding as in the end of proof of Lemma~\ref{lLocal3},
we conclude that $I=\mfp$ is a prime ideal of $R$.  Set
$H:=(f_1,\ldots,f_7)\subset I$.  Since the aforementioned resolution
of $R/I$ is minimal, $f_8\not\in H$ and $H:f_8\subsetneq R$. However,
one can check that
$f_8^2=x^2yztf_1^2-x^4f_1f_5-x^2f_2f_7+tf_5f_6+x^2f_6f_7$.  Thus
$f_8^2\in HI$ and $HI:f_8^2=R$. Therefore, $H:f_8\subsetneq HI:f_8^2$
and $I$ is not syzygetic (see \cite[Lemma~4.2]{planasCamb}).
\end{proof}

In terms of the Andr\'e-Quillen homology (see \cite{andre} and
\cite{quillen}; see also \cite{mr}, for a new and recent treatment),
and as a corollary of the Theorem, we state the following
characterization of Noetherian rings of low global dimension. Again,
just for the sake of completeness, we include item (A), shown in
general in \cite{planasPams}. 

\begin{corollary*}
Let $R$ be a Noetherian ring.
\begin{itemize}
\item[{(A)}:] $\gldim(R)\leq 1\Leftrightarrow H_2(R,S,\cdot)=0$, or
  $H_2(R,S,S)=0$, for every quotient ring $S=R/I$.
\item[{(B)}:] $\gldim(R)\leq 2\Leftrightarrow H_2(R,S,\cdot)=0$, for
  every quotient domain $S=R/I$.
\item[{(C)}:] $\gldim(R)\leq 3\Leftrightarrow H_2(R,S,S)=0$, for
every quotient domain $S=R/I$.
\end{itemize}
\end{corollary*}

Note that unlike Theorem~(B), Corollary~(B) could be deduced directly
from \cite[Corollary~2.7]{gop}, since the vanishing of the second
Andr\'e-Quillen homology $H_2(R,R/I,\cdot)$, in the Noetherian local
case, is equivalent to $I$ being generated by an $R$-regular
sequence. We give here a slightly different approach.

\begin{proof}[Proof of the Corollary]
The equivalence between the first and the third condition in
Corollary~{(A)}, follows immediately from the isomorphism
$H_2(R,R/I,R/I)\cong\ker(\alpha_2)$ and the corresponding equivalence
between the first and the third condition in Theorem~{(A)} (see, e.g.,
\cite[Corollaire~3.2]{planasMan}). Similary, Corollary~{(C)} follows
immediately from this same isomorphism and Theorem~{(C)}.

It remains to prove the first equivalence of Corollary~(A) and the
equivalence of Corollary~(B). To this end, recall that the vanishing
of $H_2(R,R/I,\cdot)$ is also equivalent to $I$ being of linear type
and $I/I^2$ being a flat $R/I$-module (see
\cite[Th\'eor\`eme~4.2]{planasMan}). Clearly, this characterization
together the corresponding ``if'' implications in Theorem~(A) and (B),
show the ``if'' implications of Corollary~(A) and (B),
respectively. Finally, note that, as said before, if $\gldim(R)\leq
1$, then every nonzero ideal $I$ of $R$ is locally principal, hence
its conormal module $I/I^2$ is $R/I$-flat. Similarly, if
$\gldim(R)\leq 2$, any nonzero prime ideal $I$ of $R$ is either
locally principal, or else maximal, hence in any case, its conormal
module $I/I^2$ is again $R/I$-flat.
\end{proof}

\begin{remark*}
If we omit the Noetherian assumption on the ring $R$, we know that
Theorem~(A) is true once we substitute $\gldim(R)$ for $\wdim(R)$
(cf. \cite{planasPams}). Note that $\wdim(R)$ can be strictly smaller
than $\gldim(R)$, for instance, if $R$ is the ring of all algebraic
integers (see, e.g., \cite[1.3~Examples]{vasconcelos}). Therefore the
``if'' implication of Theorem~(A), without the Noetherian hypothesis,
is false. This suggests that one should also replace $\gldim(R)$ for
$\wdim(R)$ in the ``if'' implication of Theorem~(B) and (C).

Just to have a flavour of the ins and outs of the non Noetherian
setting, and to start with, we show the following simpler statement:
\begin{equation}\label{nnoeth}
\text{\em If $R$ is a commutative ring of $\gldim(R)\leq 2$, then
  every prime ideal of $R$ is of linear type.}ç
\end{equation}
Indeed, since the linear type condition is local, we can suppose again
that $(R,\mfm)$ is local. Then $R$ is either a regular local ring (of
$\krulldim(R)\leq 2$), a valuation domain, or a so-called umbrella
ring (see \cite[2.2~Theorem]{vasconcelos}, for the definitions and a
proof). Concretely, it is shown that $R$ is a GCD domain with every
prime ideal different from the maximal being flat, hence of linear
type (see, e.g., \cite[Remarque~2.6]{planasMan}). As for the maximal
ideal $\mfm$, it is shown that it is either principal, generated by
two elements, or non finitely generated. In the first case, $\mfm$ is
of linear type whereas in the last case, it is shown that $R$ is a
valuation domain, hence $\wdim(R)\leq 1$, and so every ideal is of
linear type, in particular, $\mfm$ is also of linear type. Finally, if
$\mfm$ is generated by two elements, $a,b$, say, then there exists an
exact sequence $0\to R\xrightarrow{\varphi}
R^2\xrightarrow{\psi}\mfm\to 0$, with $\varphi(1)=(\alpha,\beta)$,
say, $\alpha,\beta\in R$, with $\gcd(\alpha,\beta)=1$, and
$\psi(u,v)=ua+vb$. Since $(b,-a)\in\ker(\psi)$, then there exists
$\delta\in R$, such that $a=-\delta\beta$ and $b=\delta\alpha$. Note
that $\alpha,\beta\in\mfm$, otherwise, if for instance $\alpha$ is
invertible, then $\delta=\alpha^{-1}b$ and
$a=-\delta\beta=(-\alpha^{-1}\beta)b$ and $\mfm$ would be principal.
Hence $\mfm=(a,b)R=\delta(\alpha,\beta)R
\subseteq(\alpha,\beta)R\subseteq\mfm$, and $\mfm=(\alpha,\beta)$ is
generated by the regular sequence $\alpha,\beta$, in particular,
$\mfm$ is of linear type.

We do not know whether one can substitute $\gldim(R)\leq 2$ by
$\wdim(R)\leq 2$ in \eqref{nnoeth}; neither we know if the converse of
\eqref{nnoeth} is true, even if we replace $\gldim(R)\leq 2$ by
$\wdim(R)\leq 2$. This could be a line of enquiry in future work. 
\end{remark*}

\section*{Acknowledgement} 
It is a pleasure to thank the conversations with Jos\'e M. Giral,
Javier Majadas and Bernat Plans about this subject. Most of the
calculations in this note have been done with the inestimable help of
{\sc Singular}.

{\footnotesize\sc

\noindent Departament de Matem\`atiques, Universitat Polit\`ecnica de
Catalunya. \newline Diagonal 647, ETSEIB, E-08028 Barcelona, Catalunya.
Email: {\em francesc.planas@upc.edu} }


{\small

\begin{thebibliography}{cc}
\bibitem{andre}{M. Andr\'e, Homologie des alg\`ebres commutatives. Die
  Grundlehren der mathematischen Wissenschaften, Band
  206. Springer-Verlag, Berlin-New York, 1974.}
\bibitem{bresinsky}{H. Bresinsky, On prime ideals with generic zero
  $x_i=t^{n_i}$. Proc. Amer. Math. Soc. 47 (1975), 329-332.}
\bibitem{bh}{W. Bruns, J. Herzog, Cohen-Macaulay rings. Cambridge
  Studies in Advanced Mathematics, 39. Cambridge University Press,
  Cambridge, 1993.}
\bibitem{dgps}{W. Decker, G.M. Greuel, G. Pfister, H. Sch{\"o}nemann,
  \newblock {\sc Singular} {4-1-2} --- {A} computer algebra system for
  polynomial computations.  \newblock {http://www.singular.uni-kl.de}
  (2019).}  
\bibitem{gop}{S. Goto, L. O'Carroll, F. Planas-Vilanova, Sally's
  question and a conjecture of Shimoda. Nagoya Math. J. 211 (2013),
  137-161.}
\bibitem{hsv}{J. Herzog, A. Simis, W.V. Vasconcelos, Koszul homology
  and blowing-up rings. Commutative algebra (Trento, 1981),
  pp. 79-169, Lecture Notes in Pure and Appl. Math., 84, Dekker, New
  York, 1983.}
\bibitem{huneke}{C. Huneke, Hilbert functions and symbolic
  powers. Michigan Math. J. 34 (1987), no. 2, 293-318.}
\bibitem{lam}{T.Y. Lam, Lectures on modules and rings. Graduate Texts
  in Mathematics, 189. Springer-Verlag, New York, 1999.}
\bibitem{mr}{J. Majadas, A.G. Rodicio, Smoothness, regularity and
  complete intersection. London Mathematical Society Lecture Note
  Series, 373. Cambridge University Press, Cambridge, 2010.}
\bibitem{matsumura}{H. Matsumura, Commutative ring theory. Translated
  from the Japanese by M. Reid. Cambridge Studies in Advanced
  Mathematics, 8. Cambridge University Press, Cambridge, 1986.}
\bibitem{planasMan}{F. Planas-Vilanova, Sur l'annulation du deuxi\`eme
  foncteur de (co)homologie d'Andr\' e-Quillen. Manuscripta Math. 87
  (1995), no. 3, 349-357.}
\bibitem{planasPams}{F. Planas-Vilanova, Rings of weak dimension one and
  syzygetic ideals. Proc. Amer. Math. Soc. 124 (1996), no. 10,
  3015-3017.}
\bibitem{planasCamb}{F. Planas-Vilanova, On the module of effective
  relations of a standard algebra. Math. Proc. Cambridge
  Philos. Soc. 124 (1998), no. 2, 215-229.}
\bibitem{quillen}{D. Quillen, On the (co-) homology of commutative
  rings. 1970 Applications of Categorical Algebra (Proc. Sympos. Pure
  Math., Vol. XVII, New York, 1968) pp. 65-87. Amer. Math. Soc.,
  Providence, R.I.}
\bibitem{sh}{I. Swanson, C. Huneke, Integral closure of ideals, rings,
  and modules. London Mathematical Society Lecture Note Series,
  336. Cambridge University Press, Cambridge, 2006. }
\bibitem{vasconcelos}{W.V. Vasconcelos, The rings of dimension two. 
Lecture Notes in Pure and Applied Mathematics, Vol. 22. 
Marcel Dekker, Inc., New York-Basel, 1976.}
\end{thebibliography}
}
\end{document}